\documentclass[11pt]{amsart}   	
\usepackage{geometry}             
\usepackage[dvipsnames]{xcolor}   		             		
\usepackage{graphicx}				
\usepackage{amssymb,mathrsfs}
\usepackage{amsthm,amsmath,stmaryrd}
\usepackage{tikz}
\usepackage{tikz-cd}
\usepackage{accents,upgreek,enumerate}
\usepackage[headings]{fullpage}
\usepackage{bm}
\usepackage{mathtools}
\usepackage[all]{xy}
\usepackage{caption}
\usepackage[euler-digits]{eulervm}
\usepackage[normalem]{ulem}

\usepackage{verbatim}
\usepackage{booktabs}
\usepackage{float}

\tikzset{
  commutative diagrams/.cd, 
  arrow style=tikz, 
  diagrams={>=stealth}
}

\newtheorem{innercustomthm}{Theorem}
\newenvironment{customthm}[1]
  {\renewcommand\theinnercustomthm{#1}\innercustomthm}
  {\endinnercustomthm}

  \makeatletter
\def\@tocline#1#2#3#4#5#6#7{\relax
  \ifnum #1>\c@tocdepth 
  \else
    \par \addpenalty\@secpenalty\addvspace{#2}%
    \begingroup \hyphenpenalty\@M
    \@ifempty{#4}{%
      \@tempdima\csname r@tocindent\number#1\endcsname\relax
    }{%
      \@tempdima#4\relax
    }%
    \parindent\z@ \leftskip#3\relax \advance\leftskip\@tempdima\relax
    \rightskip\@pnumwidth plus4em \parfillskip-\@pnumwidth
    #5\leavevmode\hskip-\@tempdima
      \ifcase #1
       \or\or \hskip 1em \or \hskip 2em \else \hskip 3em \fi%
      #6\nobreak\relax
    \dotfill\hbox to\@pnumwidth{\@tocpagenum{#7}}\par
    \nobreak
    \endgroup
  \fi}
\makeatother

\usetikzlibrary{calc}
\usetikzlibrary{fadings}
\usetikzlibrary{decorations.pathmorphing}
\usetikzlibrary{decorations.pathreplacing}

\newcounter{marginnote}
\DeclareMathAlphabet{\mathpzc}{OT1}{pzc}{m}{it}

\usepackage[backref=page]{hyperref}
\hypersetup{
  colorlinks   = true,          
  urlcolor     = violet,          
  linkcolor    = blue,          
  citecolor   = violet             
}

\newtheorem{theorem}{Theorem}[subsection]

\newtheorem{lemma}[theorem]{Lemma}
\newtheorem{proposition}[theorem]{Proposition}

\newtheorem{quasi-theorem}[theorem]{Quasi-Theorem}

\theoremstyle{definition}
\newtheorem{definition}[theorem]{Definition}

\newtheorem{remark}[theorem]{Remark}

\newtheorem{example}[theorem]{Example}
\newtheorem{blank remark}[theorem]{}

\newtheorem{not1}[theorem]{Notation}

\newcommand{\CC} {{\mathbb C}}

\newcommand{\PP}{\mathbb{P}}         
\newcommand{\QQ} {{\mathbb Q}}		
\newcommand{\RR} {{\mathbb R}}		
\newcommand{\ZZ} {{\mathbb Z}}		

\newcommand{\Hom}{\operatorname{Hom}}

\DeclareMathOperator{\spec}{Spec}

\DeclareMathOperator{\ch}{ch}

\newcommand{\cal}{\mathcal}

\def\cM{{\cal M}}

\def\cO{{\cal O}}

\def\trop{\mathrm{trop}}

\newcommand{\tropC}{\Gamma}

\definecolor{lightgreen}{HTML}{E8F9E2}
\definecolor{lightblue}{HTML}{ADD8E6}
\definecolor{lightcyan}{HTML}{E0FFFF}
\definecolor{lightred}{HTML}{ffcccb}
\definecolor{lightpink}{HTML}{ffe6e6}

\makeatletter
\def\blfootnote{\xdef\@thefnmark{}\@footnotetext}
\makeatother

\title{Constructions of superabundant tropical curves in higher genus}

\date{}

\author{Sae Koyama}

\address{Sae Koyama \\ Department of Pure Mathematics {\it \&} Mathematical Statistics\\
University of Cambridge, Cambridge, UK}
\email{\href{mailto:sk2050@cam.ac.uk}{sk2050@cam.ac.uk}}


\begin{document}

\maketitle

\begin{abstract}
    We construct qualitatively new examples of superabundant tropical curves which are non-realizable in genus $3$ and $4$. These curves are in $\RR^3$ and $\RR^4$ respectively, and have properties resembling canonical embeddings of genus $3$ and $4$ algebraic curves. In particular, the genus $3$ example is a degree $4$ planar tropical curve, and the genus $4$ example is contained in the product of a tropical line and a tropical conic. They have excess dimension of deformation space equal to $1$. Non-realizability follows by combining this with a dimension calculation for the corresponding space of logarithmic curves.  
\end{abstract}

\section*{Introduction}

The question of tropical realizability asks whether a given parametrized tropical curve can be realized as the troplicalization of an algebraic curve, which is important for ``lifting theorems'' that relate tropical curves to algebraic ones. This question and its variants have been studied by a number of researchers using different methods \cite{BPR16, CFPU, JR21, KatLift, MUW17, Ni09, NS06, R16, RSW17B, Sp07, tyomkin10}. 
Realizability theorems have implications for curve counting and existence theorems in Brill--Noether theory \cite{JR21, Mi03}, whilst the \emph{non}-existence of certain tropical divisors has deep implications about the non-existence of corresponding special divisors \cite{CDPE10}. It is hoped that a better understanding of which tropical curves are non-realizable will allow for wider usage of tropical techniques.

Here we consider the following question:

\begin{center}
    (Q) What do non-realizable curves ``look like''?    
\end{center}

In particular, we study the geometric conditions for a tropical curve to be \emph{superabundant}. Given a tropical curve $\Gamma$ in $\RR^r$, there is an associated deformation space $\mathsf{Def}(\Gamma)$ of parametrized tropical curves with the same combinatorial type. A tropical curve is said to be superabundant when this space of deformations is larger than expected. Realizability in the non-superabundant case was established by Cheung--Fantini--Park--Ulirsch \cite{CFPU}. However, superabundancy is not sufficient for non-realizability, and there exist superabundant combinatorial types where every curve is realizable (see Section \ref{sec:bigalg}). In general, all that is known about the realizability problem for such curves is that there is a closed, conical subset of the deformation space which parametrizes the realizable curves \cite{R16}.  

The simplest type of superabundancy for a tropical curve $\Gamma \rightarrow \RR^r$ is \emph{planar superabundancy}, where a cycle of $\Gamma$ is contained inside an affine hyperplane. This case is well-known, and realizability has been studied \cite{BPR16, KatLift, R16, RSW17B, Sp07}. Every genus $1$ tropical curve is either non-superabundant or planar superabundant. A different type of superabundancy in genus $2$ is the so-called ``triple tuning fork'' (see Figure \ref{fig:tuning-fork}). Its deformation space has dimension $3$, but the expected dimension is $2$. 

\begin{figure}[h]
    \centering
    \begin{tikzpicture}
        \filldraw [black] (-0.1, -0.05) circle (2pt);
        \filldraw [black] (0.1, 0.08) circle (2pt);
        \filldraw [black] (1, 0) circle (2pt); 
        \filldraw [black] (0, 1) circle (2pt); 
        \filldraw [black] (-0.8, -0.8) circle (2pt);
        \draw (-0.05,-0.05) -- (1, -0.05);
        \draw (-0.07,-0.05) -- (-0.07, 1);
        \draw (-0.1,-0.05) -- (-0.8,-0.75);
        \draw (0.1, 0.05) -- (1, 0.05);
        \draw (0.05, 0.05) -- (0.05, 1);
        \draw (0.1, 0.05) -- (-0.75, -0.8);
        \draw (-0.8, -0.8) -- (-1.3,-1.3);
        \draw (1, 0) -- (2, 0);
        \draw (0, 1) -- (0, 2);
    \end{tikzpicture}
    \caption{The triple tuning fork in $\RR^2$.}
    \label{fig:tuning-fork}
\end{figure}
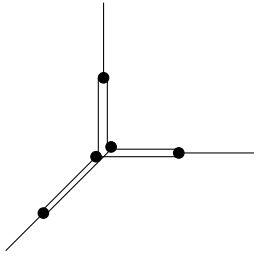

Our main result is that there exists fundamentally new types of superabundancy in genus three and four, which gives rise to tropical curves whose generic deformation is non-realizable. These are qualitatively new examples in the following sense. The notion of superabundancy can be extended to any pair $(\Gamma, \varphi)$ where $\varphi: \Gamma \rightarrow \RR^r$ is a piecewise linear map from a graph $\Gamma$,  which in general may not be balanced (see Definition \ref{def:subcurve}). A superabundant tropical curve of genus $g$ will have \emph{irreducible} superabundancy if the minimal superabundant \emph{subcurve} has genus $g$ (see Definition \ref{defirred}). Another sense in which superabundant phenomena may be new is \emph{indecomposable} superabundancy, where we ask that every projection produces a non-superabundant curve (see Defintion \ref{defindecomp}).  

Our first example is in genus $3$. In Section~\ref{sec: genus-3}, we construct a trivalent parameterized tropical curve $\varphi_3: \Gamma_3\to \RR^3$ with the following two features: 
\begin{enumerate}[(i)]
    \item the image of $\Gamma_3$ is contained in a standard tropical plane in $\RR^3$;
    \item the image of $\varphi_3$ is a curve of degree $4$ in this tropical plane\footnote{The notion of degree can be taken to be an elementary one here: the intersection of $\Gamma_3$ with the generic translate of the standard tropical plane consists of $4$ points.}.
\end{enumerate}
Moreover, the deformation space of this tropical curve has dimension 1 larger than expected. 

In Section~\ref{sec: genus-4}, we construct a trivalent parameterized tropical curve $\varphi_4:\Gamma_4\to \RR^4$ with the following features: 
\begin{enumerate}[(i)]
    \item the image of $\Gamma_4$ is contained in a product $L^\trop\times Q^\trop\subset \mathbb R^4$, where $L^\trop$ is a tropical line in $\RR^2$ and $Q^\trop$ is a tropical conic in $\RR^2$;
    \item the projections to $L^\trop$ and $Q^\trop$ are each harmonic of degree $3$. 
\end{enumerate}Again, the deformation space has a dimension 1 larger than expected. 

The \emph{core neighbourhood} of a genus $g$ tropical curve $\Gamma \rightarrow \RR^r$ is the minimal subgraph of genus $g$. A tropical curve is superabundant if and only if its core neighbourhood is (see Section \ref{sec:abundancy}). 

\begin{customthm}{A}\label{thmA}
If $\varphi: \Gamma \rightarrow \RR^g$ is a genus $g = 3$ (resp. $g = 4$) parametrised tropical curve which contains $\varphi_3$ (resp. $\varphi_4$) as a core neighbourhood, then \begin{enumerate}
    \item $\varphi$ has irreducible superabundancy;
    \item if $\varphi$ has degree\footnote{By degree here we mean the tropical curve has $d$ legs in each of the directions $e_i$ and $-\sum e_i$.} at least 5 (resp. 7), then a generic deformation of $\varphi$ is not realizable. 
\end{enumerate} 

Moreover, if $\varphi$ contains $\varphi_3$ as a core neighbourhood, then it has indecomposable superabundancy as well. 
\end{customthm}

Superabundancy is not sufficient for non-realizability since the algebraic moduli space may also have larger than expected dimension. Some degree condition is necessary to control the dimension of the moduli space of algebraic curves.

We also outline a general method for finding superabundant curves of any genus. In this process, we will see that genus $g$ superabundancy belongs to $\RR^g$, in the sense that given any superabundant genus $g$ parametrized tropical curve in 
$\RR^d$ for $d>g$, there exists a projection to $\RR^g$ whose image is superabundant (see Proposition \ref{prop1}). These fall into a more general class of rational transformations, which can be used to cut down the search space and classify all types of superabundancy in genus two and three.

The genus three and four curves have features which mirror that of canonical special linear divisors. For genus $2$, the ``triple tuning fork'' has a degree $2$ map to the tropical line\footnote{For the balancing condition to hold, the edges of infinite length must have weight $2$.}, which is a collection of three rays in directions $e_1, e_2$ and $-e_1-e_2$. Analogously, every smooth genus $2$ algebraic curve is hyperelliptic. An exact analogy to superabundant, genus $2$ tropical curves in $\RR^2$ is not possible, since the tropical curve may be contained in the intersection of planes:
$$\{x=0\} \cap \{y=0\} \cap \{x=y\}$$ 
This is a tropical line with ``linear parts extended'' (see Figure \ref{fig:tropline}). 

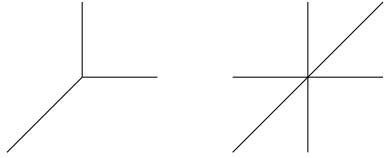
\begin{figure}[h]
    \centering
    \begin{tikzpicture}
        \draw (-1, 1) -- (-1, 0);
        \draw (-1, 0) -- (0, 0);
        \draw (-2, -1) -- (-1, 0);
        \draw (1, 0) -- (3, 0); 
        \draw (1, -1) -- (3, 1);
        \draw (2, -1) -- (2, 1);
    \end{tikzpicture}
    \caption{The tropical line and the tropical line with linear parts extended.}
    \label{fig:tropline}
\end{figure}

\begin{customthm}{B} \label{thmB}
Let $\varphi:\Gamma\to \RR^2$ be a parameterized tropical curve of genus $2$. If $(\Gamma,\varphi)$ is superabundant then either:
	\begin{enumerate}[(i)]
		\item $(\tropC,\varphi)$ is planar superabundant, i.e. a cycle of $\Gamma$ is contained in an affine line in $\RR^2$;
		\item $(\tropC,\varphi)$ may be transformed by a rational dilation into a  tropical curve $(\tropC',\varphi')$, where $(\tropC', \varphi')$ has three segments contained in the lines $\{x=0\}, \{y=0\}$, and $\{x=y\}$ respectively. 
		\end{enumerate}
Conversely, any tropical curve satisfying one of the two conditions above is superabundant. 		
\end{customthm}

A \emph{segment} is defined in Section \ref{sec:abundancy} and is roughly the parts of the underlying graph that make up the cycles. This second condition is a modification of the tuning fork, which demonstrates some kind of ``hyperelliptic'' behaviour. The result mirrors that of special linear series of genus $2$ algebraic curves, which are either canonical or trivial. Extrapolating from our genus three and four examples, we conjecture the existence of ``canoncial'' superabundancy for higher genus curves, with excess dimension of the deformation space equal to $1$.

\subsection*{Motivation and Prior Work} 

Tropical realizability is a fundamental question when using tropical techniques to study algebraic objects. Mikhalkin first studied realizability in the case of smooth (in particular, non-superabundant), nodal, planar curves and used this to obtain results in curve counting \cite{Mi03}. Tropical analogues of divisors have led to a new proof of the Brill--Noether theorem by showing that certain algebraic objects cannot exist because of the non-existence of the corresponding tropical object \cite{CDPE10}. In the other direction, realizability results can be used to transfer tropical constructions to algebraic ones and this approach has led to new results  \cite{JR21}. Another variant of the realizability problem is that of tropical double ramification cycles \cite{UZ19}. 

One approach to the realizability problem is logarithmic deformation theory, which was laid out by 
Nishinou--Siebert \cite{NS06}. The realizability results of Cheung--Fantini--Park--Ulirsch expanded on this setup \cite{CFPU}. In genus $0$, all tropical curves are non-superabundant, hence realizable. In genus $1$, much work has been done to create a full picture, including other approaches such as non-archimedean geometry \cite{BPR16}. The superabundant curves are planar superabundant, and realizable if and only if they satisfy the \emph{well-spacedness} condition \cite{KatLift, R16, RSW17B, Sp07}.

Extensions of these techniques to some higher genus cases have been studied \cite{Ni09}. A promising direction is via Hurwitz covers \cite{ABBR15, CMR14a}, which reduces certain cases to tropical modifications of realizable ones. 
Superabundancy phenomena propagate to higher genera so lower genus cases can give families of realizable or non-realizable curves. 
However, our results suggest that there are new types of superabundacy in every genus. For higher genera, even the combinatorial conditions for superabundancy become more complicated and are not very well understood.  

We speculate that the ``canonical'' superabundant curves identified up to genus $4$ will have realizability conditions analagous to the well-spacedness condition. In particular, the excess dimension of the deformation space is $1$, so we expect these to have a single condition on edge lengths determining realizability. There may be an interesting connection with the realizability of tropical \emph{canonical} divisors which have been studied by M{\"o}eller--Ulirsch--Werner \cite{MUW17}.

\subsection*{Outline}

Background is given in Section \ref{section1}, with discussion of deformation spaces and superabundancy. Section \ref{section2} discusses superabundancy in genus $2$ and $3$. In Section \ref{section3}, we demonstrate a genus $4$ superabundant example. In Section \ref{sec:4}, we give an example of a superabundant tropical curve which is realizable, and show that the genus $3$, $4$ examples give non-realizable curves. We work over an algebraically closed field $k$ of characteristic $0$ throughout, but most the results will extend to positive characteristic. 

\subsection*{Acknowledgements} The main body of this paper was developed during a project for the Cambridge Summer Research in Mathematics programme in 2021. I thank my supervisor, Dhruv Ranganathan, for his continued support and insightful conversations, as well as everyone who was involved with the programme for creating an enjoyable and productive atmosphere. Many thanks to the anonymous referee who gave a careful reading and many helpful comments.

\section{Superabundant tropical curves} \label{section1}

\subsection{Basic definitions} 

We summarise some basic notions, largely following definitions and termininology given in \cite{CFPU}. 

\begin{definition} 
     An \emph{abstract tropical curve} $\Gamma = (G, m, \ell)$ consists of the following data: a (finite) graph $G$, a \emph{marking function} $m:  \{1, \ldots n\} \rightarrow V(G)$ to the vertex set of $G$, and a \emph{length function} $\ell:E(G) \rightarrow \RR_{>0}$ from the edge set of $G$. 
     
     The curve $\Gamma$ can be given the structure of a metric space by giving each edge $e$ of $G$ the positive real edge-length $\ell(e)$ to form $G^{\text{met}}$, then defining

     $$\Gamma^{\text{met}}=(G^{\text{met}} \amalg \coprod_{i \in \{1, \ldots n\} }\RR_{\geq0}) / \sim$$
     where the equivalence relation glues each $0$ to the vertex $m(i)$. The edges of infinite length corresponding to markings are  called \emph{legs}. We will often refer to $\Gamma^{\text{met}}$ as $\Gamma$. 
\end{definition}

As a metric space, the tropical curve $\Gamma$ does not remember the underlying graph structure, in particular the existence of valence two vertices. However, every point on $\Gamma$ has a well-defined valence and for any $\Gamma$ we can pick out a finite vertex set. We often consider tropical curves whose vertices have valence $2$ or $3$. Here we do not include genus at vertices, although higher genus vertices are of interest - it has been shown that realizable tropical curves with higher genus vertices give rise to superabundant genus zero vertex tropical curves \cite{MR19}. 

The identification of the edges of $\Gamma$ with real intervals allows us to define \emph{piecewise linear functions} on $\Gamma$ as continuous maps $\Gamma \rightarrow \RR^r$ which restrict to affine linear functions on the edges. 

\begin{definition}
    A \emph{parametrized tropical curve} is a pair $(\Gamma, \varphi)$ where $\Gamma$ is an abstract tropical curve, and $\varphi$ is a piecewise linear function with integer slopes, satisfying the following balancing condition: at every point in $\Gamma$, the sum of the directional derivatives of $\varphi$ is 0. 
\end{definition}

Note that the balancing conditions implies that in the neighbourhood of a $2$--valent vertex, the map $\varphi$ is linear. 

\begin{remark}
    Sometimes one imposes that each edge $e$ has a \emph{primitive} integral slope $u_e$. Then the directional derivative of $\varphi$ along each edge is an integer multiple $w_e$ of $u_e$, which is the \emph{weight} associated to this edge. We will mostly not be concerned with this. 
\end{remark}

Let $\varphi: \Gamma\to \RR^r$ be a parameterized tropical curve, with $\Gamma$ trivalent. Its \textit{combinatorial type} consists of the data obtained by dropping the edge lengths.  More formally:

\begin{enumerate}[(i)]
\item The finite graph (with one-valence vertices removed) $|\Gamma|$ underlying $\Gamma$ obtained from the minimal vertex set. 
\item For each directed edge of $|\Gamma|$, the edge direction in $\RR^r$. Equivalently, the tuple of slopes of $\varphi$ along the directed edge. The tuple is required to be negated when the orientation is reversed.
\end{enumerate}

\subsection{The abundancy matrix}  \label{sec:abundancy}

Given a combinatorial type $\Theta$, the group $\RR^r$ acts on the set of all parametrised tropical curves with this combinatorial type by translation. The quotient set is a rational polyhedral cone $\sigma_\Theta$ \cite{CFPU,R16}. The dimension of this cone can be calculated using the abundancy map \cite[Section 4]{CFPU}. 

\begin{definition} \cite[Definition 4.1]{CFPU}
The \textit{abundancy map} associated to a tropical curve $\Gamma$ is the linear map
\[
\Phi_\Gamma: \RR^{E}\to \Hom(H_1(\Gamma),\RR^r)
\] 
sending a tuple of real numbers associated to edges to their weighted displacement in $\RR^r$ given by the formula:
\[
(\ell_e)\mapsto \left(\sum_{e\in E(G)} a_e[e] \mapsto \sum_{e\in E(G)} \ell_e a_e\vec e\right)
\]
where $\vec e$ is the direction vector of the edge $e$.
\end{definition}

The abundancy map captures a failure of cycles to close up when arbitrary edge lengths are chosen. If the abundancy map is surjective, then $\Gamma$ is non-superabundant, and the dimension of the moduli space can be calculated as the dimension of the kernel of this map \cite{CFPU, KatLift, Mi03}. Roughly, one expects a codimension $n$ condition for each cycle in $\Gamma$. We reproduce this result by choosing basis for the cycle space and writing in coordinates. 

\begin{proposition}
	Fix a trivalent combinatorial type $\Theta$ as above, then
	\begin{equation*}
		\dim \sigma_\Theta \geq n+(r-3)(1-g)-r
	\end{equation*}
\end{proposition}

\begin{proof}
	Suppose $\Gamma$ has $b$ edges, labelled $e_1,\ldots,e_b$. The map $\varphi$ is a vector valued linear function on each edge, and fixing $e$ as such an edge, we can write $$\varphi(x) = \mathbf{w_e}x+\mathbf{c}$$ on $e$. Choose a basis of $g$ simple cycles in $\Gamma$. Let 
	$$\eta_{g,i} = \begin{cases}
	1 &\text{ if $e_i$ in cycle $g$ and direction preserved} \\
	-1 &\text{ if $e_i$ in cycle $g$ and direction reversed} \\
	0 &\text{ otherwise } 
	\end{cases}$$
	We define the \emph{abundancy matrix} as the following block matrix:
	$$ K = 
		\begin{pmatrix}
		\eta_{1,1} \mathbf{w}_{e_1} & \ldots & \eta_{1,b} \mathbf{w}_{e_b}\\
		\vdots & \ddots & \vdots \\
		\eta_{g,1} \mathbf{w}_{e_1} & \ldots & \eta_{g,b} \mathbf{w}_{e_b}
	\end{pmatrix}
	$$
	We wish to determine whether a given point $\mathbf{l} \in \RR^b$ lies is an assignment of edge lengths for which the continuous map $\varphi$ exists. Since every cycle must close up, we require that for every row $\mathbf{k}_j$ of K, the dot product $\mathbf{k}_j \cdot \mathbf{l}$ must vanish. Since we started with a basis of simple cycles, this is a necessary and sufficient condition.
	 
	There are $rg$ rows in the matrix, hence the dimension of $\sigma_\Theta$ is at most $b-rg$. The result follows from noticing that by an Euler characteristic calculation $b = n+3g-3$. 
\end{proof}

If $\dim \sigma_\Theta$ is larger than this lower bound, we say the tropical curve with combinatorial type $\Theta$ is \emph{superabundant}. This occurs precisely when the rows of $K$, namely the vectors $\mathbf{k}_j$ in the proof, are not linearly independent. 

Superabundancy is determined only by information on the edges which form cycles. To simplify the picture, it will be useful to pass to subgraphs.

\begin{definition}\label{def:subcurve}
	Suppose $(\Gamma, \varphi)$ is a parametrised tropical curve with $\Gamma = (G, m, \ell)$. Then a \emph{subcurve} of $\Gamma$ is \begin{enumerate}[(i)]
	    \item an abstract tropical curve $\Gamma' = (G', m', \ell')$, where $G'$ is a subgraph of $G$, $m' : I \subset \{1, \ldots n\} \rightarrow V(G')$ is a restriction of $m$, and $\ell': E(G') \rightarrow \RR_{>0}$ is a restriction of $\ell$;
     \item a (piecewise linear) function $\varphi' : \Gamma' \rightarrow \RR^r$ such that $\varphi' = \varphi|_{\Gamma'}$.
	\end{enumerate}
 
    Here $\Gamma'$ can be considered as a subspace of $\Gamma$ in the natural way. 
\end{definition}

A subcurve is not necessarily a tropical curve, since the balancing condition may not hold. Its deformation space can nonetheless be defined exactly as above.

\begin{definition} 
	A subcurve of a tropical curve in $\RR^r$ is \emph{superabundant} if the dimension of the deformation space (modulo translations) is greater than $b-rg$, where $b$ is the number of bounded edges and $g$ is the genus of the subcurve. 
\end{definition}

The superabundancy condition can be seen geometrically by considering a simplification procedure of the underlying graph $G$ of the tropical curve. After choosing a basis of simple cycles for the first homology of $G$, delete edges that are not in any basis of cycles and then delete any vertices that now have valence $0$. We may now pass to the underlying topological space and then equip it with its \textit{minimal} graph structure. The result is called the \textit{smoothing} of the graph. It is straightforward to see that this is independent of the choice of basis for the cycle space (see Figure~\ref{fig:smoothing}).

\begin{figure}[h]
    \centering
    \begin{tikzpicture}[scale=0.6]
    \coordinate (1) at (-5, 0);
    \coordinate (2) at (-6.5, 0);
    \coordinate (3) at (-7.5, 0.9);
    \coordinate (4) at (-9, 2);
    \coordinate (5) at (-8, -1);
    \coordinate (6) at (-6.8, -1);
    \coordinate (7) at (-5.5, -2.5);
    \coordinate (8) at (-4, -1);
    \coordinate (a) at (-4.5, 1) ;
    \coordinate (b) at (-8, 2) ;
    \coordinate (c) at (-9, -2) ;
    \coordinate (d) at (-5.5, -4) ;
    \coordinate (e) at (-3, -1) ;

    \draw[blue] (1) -- (2) ;
    \draw[red] (2) -- (3) ;
    \draw[red] (3) -- (5) ;
    \draw[red] (5) -- (6) ;
    \draw[blue] (6) -- (7) ;
    \draw[blue] (7) -- (8) ;
    \draw[blue] (8) -- (1) ;
    \draw[red] (2) -- (6) ;
    \draw (1) -- (a) ;
    \draw (3) -- (b) ;
    \draw (5) -- (c) ;
    \draw (7) -- (d) ;
    \draw (8) -- (e) ;

    \foreach \x in {(1), (2), (3), (5), (6), (7), (8)}{
        \fill \x circle[radius=2pt];
    }
    
    \draw[red] (2,-1) arc (0:180:1);
    \draw[blue] (0,-1) arc (180:360:1);
    \draw[red] (2, -1) -- (0, -1);
    \fill (0,-1) circle[radius=2pt];
    \fill (2,-1) circle[radius=2pt];

    \draw[->] (-2, -1) -- (-1, -1);
\end{tikzpicture}
    \caption{Smoothing of a graph.}
    \label{fig:smoothing}
\end{figure}
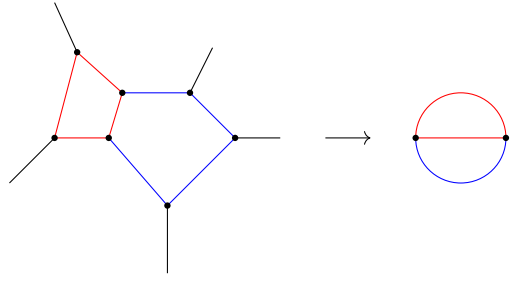

Edges that are not in any cycle of the basis cycles correspond to columns of the matrix $K$ which contain only zeros, which may be disregarded. Group the remaining edges so that each group is mapped to the same edge in the smoothed graph - call these \emph{segments}. For each segment, we can choose the directions of the edges to be consistent, i.e. they form a directed path. So for all of the edges in this segment, the associated $\eta_{i,j}$ is now the same. Swapping columns does not change the linear independence of the rows, so we may rewrite $K$ as the following block matrix:
\begin{equation}
K = 
\begin{pmatrix}
\eta_{1,1}'A_1 & \ldots & \eta_{1,t}' A_t\\
\vdots & \ddots & \vdots \\
\eta_{g,1}'A_1 & \ldots & \eta_{g,t}'A_t
\end{pmatrix}
\end{equation} 

The columns of $A_j$ are the direction vectors of edges in segment $j$ and $\eta_{i,j}'$ indicates whether the $j$th segment is in cycle $i$ (with the sign determined by the direction).

Superabundancy is equivalent to the existence of $\lambda^{(i)} \in \ZZ^r$, for $i = 1,\ldots,g$, not all zero, such that 
$$\sum_{i=1}^g\sum_{j=1}^r \lambda^{(i)}_j k_{(i-1)r+j, s} = 0$$ 
for all $s$, where $k_{i,j}$ are the entries $K$. In other words, the edges of segment $j$ are contained in the affine subspace 
$$\langle \sum_{i=1}^g \eta_{i,j}' \lambda^{(i)} \rangle^\perp + \mathbf{c}$$ 
For example, we may have that $\lambda^{(1)} \neq 0$ and all other $\lambda^{(i)}$ are 0. Then the first cycle lies inside an affine translation of the $r-1$ dimensional hyperplane $\langle \lambda^{(1)} \rangle^\perp$, so we have \emph{planar superabundance}.

\begin{definition}
    Given a parametrised tropical curve $\varphi: \Gamma \rightarrow \RR^r$ of genus $g$, its \emph{core neighbourhood} is the minimal subcurve $\varphi': \Gamma' \rightarrow \RR^r$ with genus $g$. 
\end{definition}

The core neighbourhood is obtained after deleting any edges that are not in any basis of cycles and vertices that then have valence $0$. From the discussion above, we obtain:

 \begin{lemma}
     A parametrised tropical curve $(\Gamma, \varphi)$ is superabundant if and only if its core neighbrourhood is superabundant. 
 \end{lemma}

\subsection{Rational Transformations} The group $GL_n(\ZZ)$ naturally acts by post-composition on the set of parameterized tropical curves in $\RR^r$. More generally, we consider \emph{rational} transformations. Given a matrix $Q$ in $GL_n(\QQ)$ and a tropical curve $(\Gamma,\varphi)$, the composite map
\[
\Gamma\to \RR^r\xrightarrow{Q} \RR^r
\]
need not be a parameterized tropical curve, since it may not have integral slopes. However, by changing the integral structure on $\Gamma$, it can be made into one. 

Given $\Gamma$ and an edge $e$ in it, a \textit{rational dilation of $e$} changes the metric on $e$ by multiplying the length by a positive rational number. A rational dilation of a leg of $\Gamma$ is defined analogously. A \textit{rational dilation of $\Gamma$} is a sequence of rational dilations of edges and legs. It produces a homeomorphic graph $\Gamma'$, with a different integral/metric structure. 

\begin{proposition}
	Consider the linear transformation $T:\RR^r \rightarrow \RR^s$, given by $\mathbf{x} \mapsto Q\mathbf{x}+ \mathbf{b}$ where $Q \in M_{s,r}(\QQ), \mathbf{b} \in \RR^s$. For a parameterized tropical curve $(\Gamma,\varphi)$, let $\varphi^* = T \circ \varphi$. Then there exists a parametrized tropical curve $(\Gamma',\varphi')$ such that $\Gamma'$ is a rational dilation of $\Gamma$ and $\varphi'$ is obtained by composing with the resulting homeomorphism. Moreover, if $ Q \in GL_r(\QQ)$, the deformation space of $(\Gamma',\varphi')$  coincides with that of $(\Gamma,\varphi)$ as a real cone, i.e. ignoring integral structures.
\end{proposition}

\begin{proof}
	Since $\varphi$ is continuous and linear on the edges of $\Gamma$, and $T$ is linear (hence continuous), the map $\varphi^*$ is also a continuous, piecewise linear function. 
	On edge $e$, the map $\varphi^*$ has the form $Q \mathbf{w}x+Q\mathbf{c}+ \mathbf{b}$. For this to be a tropical parametrization, we require $Q \mathbf{w}$ to have integer elements. Let $m$ be the lowest common multiple of the denominators of the entries of $Q \mathbf{w}$. We can multiply $Q\mathbf{w}$ by $m$ and divide the lengths of the edges of $\Gamma$ by $m$ to preserve the image. Let the resulting piecewise linear function and abstract tropical curve be $\varphi'$ and $\Gamma'$. 
	
	It now suffices to check the balancing condition. Consider a vertex $v$ with edges $e_1,\ldots,e_t$ adjacent to it, with outgoing slopes $\mathbf{w}_1,\ldots,\mathbf{w}_t$. The new slopes are $mQ\mathbf{w}_1,\ldots, mQ \mathbf{w}_t$. Have 
 $$mQ\mathbf{w}_1 + \ldots + mQ \mathbf{w}_t  = mQ (\mathbf{w}_1+\ldots+\mathbf{w}_t) = 0$$ 
 since $\mathbf{w}_1+\ldots+\mathbf{w}_t=0$, as required. 
	
	 Note that $\Gamma'$ has the same graph structure as $\Gamma$. Suppose $\ell$ lies in the deformation space for $(\Gamma,\varphi)$. We obtain the superabundancy matrix for $(\Gamma', \varphi')$ by multiplying all slopes by $mQ$:
	$$ K = 
	\begin{pmatrix}
	m\eta_{1,1}Q \mathbf{w}_{e_1} & \ldots & m\eta_{1,b}Q \mathbf{w}_{e_b}\\
	\vdots & \ddots & \vdots \\
	m\eta_{g,1}Q \mathbf{w}_{e_1} & \ldots & m\eta_{g,b}Q \mathbf{w}_{e_b}
	\end{pmatrix}
	$$
	By factorising, we see that 
 $$\ell_1 m\eta_{i,1} Q\mathbf{w}_{e_1} + \ldots + \ell_b m\eta_{i,b} Q\mathbf{w}_{e_b}= mQ (\ell_1 \eta_{i,1} \mathbf{w}_{e_1}\ldots+\ell_b \eta_{i,b}\mathbf{w}_{e_b})=0$$
 for $i=1,\ldots,g$, and so $\ell \in \mathcal{M}_{(\Gamma', \varphi')}$, the kernel of the associated abundancy map. 
	Hence $\mathcal{M}_{(\Gamma,\varphi)} \subset \mathcal{M}_{(\Gamma',\varphi')}$. If $Q \in GL_r(\QQ)$ then the process is reversible; the kernels of the two matrices coincide, so we ensure that the two deformation spaces coincide.
\end{proof}

To understand genus $g$ superabundancy, it suffices to understand it in $\RR^d$ where $d \leq g$.

\begin{proposition} \label{prop1}
    Given any superabundant genus $g$ parametrised tropical curve in $\RR^d$ for $d>g$, we can find a projection to $\RR^g$ whose image (perhaps after applying rational dilation) is superabundant. 
\end{proposition}

\begin{proof}
    The idea is to ``project onto the obstruction''. Suppose we have a parametrised tropical curve $(\Gamma, \varphi)$ of genus $g$ which is superabundant. Superabundancy is equivalent to the existence of $\lambda^{(1)}$, \ldots, $\lambda^{(g)} \in \ZZ^r$ with segments $i = 1, \ldots, t$ contained in 
    $$\langle \sum_{i=1}^g \eta_{i,j}' \lambda^{(i)} \rangle^\perp + \mathbf{c}$$ 
    
    We can consider an orthogonal projection $\pi$ on to the span of $\lambda^{(i)}$, which is a space with dimension at most $g$. The affine subspaces 
    
    $$\langle \sum_{i=1}^g \eta_{i,j}' \lambda^{(i)} \rangle^\perp + \mathbf{c}$$ 
    are mapped to 
    $$\langle \sum_{i=1}^g \eta_{i,j}' \lambda^{(i)} \rangle^\perp + \pi(\mathbf{c})$$ 
    So the resulting curves satisfy the conditions for superabundancy. 
\end{proof}

\subsection{Defining new phenomena}

We clarify what we mean by ``new'' types of superabundancy. There are two ways in which superabundant curves may easily be constructed from other superabundant curves. Firstly, if a genus $g$ curve contains a genus $g' < g$ subcurve which is superabundant, then it is also superabundant. Secondly, given a superabundant curve in $\RR^r$ with segments contained in affine hyperplanes with normals $\lambda^{(i)}$ as in the discussion above, we can consider the product $\RR^r \times \RR^s$, and extend the affine hyperplanes by taking the preimage of the projection. Any curve in $\RR^{r+s}$ with segments contained in these extended affine hyperplanes are also superabundant. In this case, we have a projection onto a superabundant curve in $\RR^r$. 

\begin{definition}\label{defirred}
	A superabundant tropical curve $(\Gamma, \varphi)$ has \emph{irreducible superabundacy} if it does not contain a subcurve of lower genus that is superabundant. Otherwise, say $(\Gamma, \varphi)$ has \emph{reducible superabundancy}.
\end{definition}

\begin{definition}\label{defindecomp}
        A superabundant tropical curve $(\Gamma, \varphi)$ for $\varphi: \Gamma \rightarrow \RR^r$ has \emph{decomposable superabundancy} if there exists a linear map $A: \RR^r \rightarrow \RR^{r'}$ with $r' < r$, such that a rational dilation of $(\Gamma, A\circ \varphi)$ is superabundant. Otherwise, say it has \emph{indecomposable superabundancy}.
\end{definition}


Superabundancy is determined by the basis of cycles, and so a parametrized tropical curve has irreducible or indecomposable superabundancy if and only if its core neighbourhood does. 

\begin{remark}
    These two notions are not equivalent, nor does one imply the other. For example, we can consider a modification of the triple tuning fork in $\RR^2$ (Figure \ref{fig:tuning-fork}), which is a triple tuning fork in $\RR^3$, were we have ``pulled apart'' each fork (see Figure \ref{fig:tuning-fork-mod}). No cycle is contained in a plane so it has irreducible superabundancy, but the projection to $\RR^2$ is the triple tuning fork in $\RR^2$, so is still superabundant. Thus it does not have indecomposable superabundancy.

    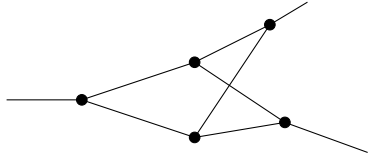
\begin{figure}[h]
        \centering

        \begin{tikzpicture}
        \filldraw [black] (0,0) circle (2pt);
        \filldraw [black] (0,-1) circle (2pt);
        \filldraw [black] (-1.5,-0.5) circle (2pt);
        \filldraw [black] (1,0.5) circle (2pt);
        \filldraw [black] (1.2,-0.8) circle (2pt);
        \draw (0,0) -- (1,0.5);
        \draw (0,0) -- (1.2,-0.8);
        \draw (0,0) -- (-1.5, -0.5);
        \draw (0,-1) -- (1,0.5);
        \draw (0,-1) -- (1.2,-0.8);
        \draw (0,-1) -- (-1.5, -0.5);
        \draw (1,0.5) -- (1.5, 0.8);
        \draw (1.2, -0.8) -- (2.3, -1.2);
        \draw (-1.5, -0.5) -- (-2.5, -0.5);
        \end{tikzpicture}

        \caption{The triple tuning fork in $\RR^3$ has irreducible superabundancy but not indecomposable superabundancy.}
        \label{fig:tuning-fork-mod}
    \end{figure}
    
    Conversely, if a tropical curve in $\RR^2$ contains a triple tuning fork in $\RR^2$ and a full-rank cycle (see Figure \ref{fig:trop-curve-2}), then this has indecomposable superabundancy but not irreducible superabundancy. 

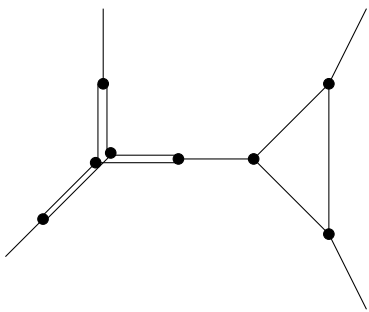
\begin{figure}[H]
    \centering
    \begin{tikzpicture}
        \filldraw [black] (-0.1, -0.05) circle (2pt);
        \filldraw [black] (0.1, 0.08) circle (2pt);
        \filldraw [black] (1, 0) circle (2pt); 
        \filldraw [black] (0, 1) circle (2pt); 
        \filldraw [black] (-0.8, -0.8) circle (2pt);
        \draw (-0.05,-0.05) -- (1, -0.05);
        \draw (-0.07,-0.05) -- (-0.07, 1);
        \draw (-0.1,-0.05) -- (-0.8,-0.75);
        \draw (0.1, 0.05) -- (1, 0.05);
        \draw (0.05, 0.05) -- (0.05, 1);
        \draw (0.1, 0.05) -- (-0.75, -0.8);
        \draw (-0.8, -0.8) -- (-1.3,-1.3);
        \draw (1, 0) -- (2, 0);
        \draw (0, 1) -- (0, 2);
        \filldraw[black] (2, 0) circle (2pt);
        \filldraw[black] (3, 1) circle (2pt);
        \filldraw[black] (3, -1) circle (2pt);
        \draw (2, 0) -- (3, 1);
        \draw (2, 0) -- (3, -1);
        \draw (3, 1) -- (3, -1);
        \draw (3, 1) -- (3.5, 2);
        \draw (3, -1) -- (3.5, -2);
        
    \end{tikzpicture}
    \caption{A tropical curve in $\RR^2$ with indecomposable, but not irreducible superabundancy.}
    \label{fig:trop-curve-2}
\end{figure}

    However, we can obtain the notion of a ``minimal'' superabundant curve. 
    
\end{remark}

\begin{proposition}
    If $\varphi: \Gamma \rightarrow \RR^r$ has irreducible superabundancy, then any projection that is superabundant also has irreducible superabundancy. 
\end{proposition}

\begin{proof}
    Suppose that for a projection $\pi$, the resulting curve $\varphi': \Gamma \rightarrow \RR^{r'}$ is superabundant of reducible type. Then there is a subcurve of strictly lower genus whose segments are contained in affine hyperplanes with configurations giving rise to the superabundancy, say with normals $\lambda^{(i)}$. The preimage of these segments in $\varphi: \Gamma \rightarrow \RR^r$ must be contained in affine hyperplanes with normals $\tilde{\lambda}^{(i)}$, where $\tilde{\lambda}^{(i)}$ are obtained from $\lambda^{(i)}$ under the inclusion $\RR^{r'} \rightarrow \RR^r$ into the first $r'$ coordinates. But the $\tilde{\lambda}^{(i)}$ satisfy the same relations that the $\lambda^{(i)}$ do, and hence give rise to superabundancy in a subcurve with strictly lower genus. 
\end{proof}

\begin{remark}
    In practice, irreducible superabundancy can be checked by removing edges one at a time and checking that the resulting subcurve is not superabundant. Note that a superabundant curve \emph{must} have all of its segments contained in affine hyperplanes. Removing an edge may cause this to fail, in which case the resulting subcurve cannot be superabundant. An example is shown in Figure \ref{fig:remove-edge}. 

    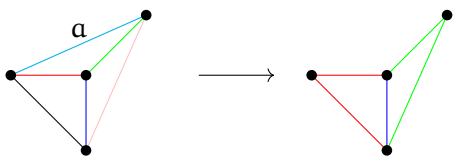
\begin{figure}[h]
        \centering
        \begin{tikzpicture}
            \coordinate (1) at (0,0);
            \coordinate (2) at (-1, 0);
            \coordinate (3) at (0, -1);
            \coordinate (4) at (0.8, 0.8);
            \coordinate (5) at (1.5, 0);
            \coordinate (6) at (2.5, 0);
            \coordinate (7) at (4,0);
            \coordinate (8) at (3, 0);
            \coordinate (9) at (4, -1);
            \coordinate (10) at (4.8, 0.8);
    
            \draw[red] (1) -- (2); 
            \draw[blue] (1) -- (3); 
            \draw[green] (1) -- (4);
            \draw[black] (2) -- (3);
            \draw[pink] (3) -- (4);
            \draw[cyan] (4) -- (2);
            \draw[red] (7) -- (8); 
            \draw[blue] (7) -- (9); 
            \draw [red] (8) -- (9);
            \draw [green] (7) -- (10);
            \draw [green] (9) -- (10);
            \node at (-0.1,0.6) {$a$};

            \draw[->] (5) -- (6);
    
    \foreach \x in {(1), (2), (3), (4), (7), (8), (9), (10)}{
        \fill \x circle[radius=2pt];
        }

        \end{tikzpicture}
        \caption{Removing edge $a$ forces segments to span $\RR^2$, so the resulting subcurve is not superabundant.}
        \label{fig:remove-edge}
    \end{figure}

    Similarly, any projection that causes a segment to not be contained in an affine hyperplane is immediately ruled out when checking for indecomposable superabundancy.
\end{remark}

\section{Superabundant curves of genus $2$ and $3$}
\label{section2}

\subsection{Genus $2$}

We prove Theorem \ref{thmB}. The strategy is to consider all possible smoothed graphs and construct affine hyperplanes where segments should lie. First note that the only superabundancy in genus $1$ is  planar superabundancy. We now look for irreducible genus $2$ superabundancy. 

\begin{remark}
    If $(\Gamma, \varphi)$, a parametrized tropical curve of genus $g$, has corresponding smoothed graph with vertex connectivity less than or equal to 1, then the two cycles are independent of one another. They each reduce to a lower genus case. Hence, if $(\Gamma, \varphi)$ has irreducible superabundancy, then the corresponding smoothed curve has vertex connectivity at least $2$. 
\end{remark}

\begin{theorem}
    A genus $2$ parametrized tropical curve has irreducible superabundancy if and only if its segments are contained (up to translation) in hyperplanes $H_1, H_2, H_3$, such that $H_1, H_2, H_3$ have normals $\lambda^{(1)}, \lambda^{(2)}$, $\lambda^{(1)}+ \lambda^{(2)}$ respectively, with $\lambda^{(1)}, \lambda^{(2)}$ linearly independent. 
\end{theorem}

\begin{proof} 

By the proceeding remark, the only possible smoothed curve is the $3$-regular graph on two vertices. Label these segments $1$, $2$, and $3$. A choice of the corresponding abundancy matrix is of the form
\begin{equation*}
	K = \begin{pmatrix}
		A_1 & 0 & A_3 \\
		0 & A_2 & A_3 \\
	\end{pmatrix}
\end{equation*}

Have $\lambda^{(1)}$ and $\lambda^{(2)}$ as before. Assuming we have irreducible superabundancy, they are non-zero. Segments $1, 2$ and $3$ are contained in hyperplanes with normal vectors $\lambda^{(1)}$,  $\lambda^{(2)}$, and $\lambda^{(1)} + \lambda^{(2)}$ respectively. 

If $\lambda^{(1)}$ and $\lambda^{(2)}$ are linearly dependent, we would again have planar superabundance. In fact, all the cycles are contained in a hyperplane, and this is reducible. The vertices where the segments meet must be contained in the intersection of the planes containing the segments. Hence these two vertices must be at the same point in $\RR^2$. Translating this point to the origin gives the result. 
\end{proof}

Apply a rational transformation as above which sends the planes with normals $\lambda^{(1)}$ and $\lambda^{(2)}$ to planes in $\RR^2$ with normals $\mathbf{e}_1,\mathbf{e}_2$ as in the statement of Theorem \ref{thmB}. This maps the superabundant curve with its trees removed onto a tropical line with linear parts extended, and segments contained in the lines $\{x=0\}, \{y=0\}$, and $\{x=y\}$.

The converse is clear, since the superabundancy matrix fails to have maximal rank, and Theorem \ref{thmB} follows. 

\subsection{Genus $3$}
\label{sec: genus-3}

We start with our main example. 

\begin{example} \label{Exa:Genus3}
    We constuct a subcurve $\varphi_3': \Gamma_3' \rightarrow \RR^3$ such that any tropical curve $\varphi: \Gamma \rightarrow \RR^3$ which contains $\varphi_3'$ as a core neighbourhood has irreducible and indecomposable superabundancy. The underlying smoothed graph is $K_4$. It is contained in a \emph{tropical plane}, which is the fan with the following top-dimensional cones: \begin{enumerate}[(i)]
        \item the positive span of $\{e_i, e_j \}$, where $e_i$ are standard basis vectors in $\RR^3$;
        \item the positive span of $\{u, e_i\}$, where $u = -(e_1 + e_2+ e_3)$.  
    \end{enumerate} 
    
    A picture is given in Figure \ref{fig:phi3} and the positions of the labelled vertices in Table \ref{tab:phi3_pos}.

    \newpage
    
    \begin{figure}[h]
        \centering
        \begin{tikzpicture}
            \coordinate (1) at (0, 0); 
            \coordinate (2) at (0, 2.5);
            \coordinate (3) at (2, 0.6);
            \coordinate (4) at (2.5, -1);
            \coordinate (5) at (-2, -1.3);

            \filldraw[color=lightgreen] (1) -- (2) -- (3) -- cycle;
            \filldraw[color=lightcyan] (1) -- (3) -- (4) -- cycle;
            \filldraw[color=lightblue] (1) -- (4) -- (2) -- cycle;
            \filldraw[color=lightred] (1) -- (5)-- (2) -- cycle;
            \filldraw[color=lightpink] (1) -- (5) -- (4) -- cycle;
            
            \draw[->] (1) -- (2) node[at end, above] {$e_1$};
            \draw[->, dashed] (1) -- (3) node[at end, above right] {$e_3$};
            \draw[->] (1) -- (4) node[at end, below right] {$e_2$};
            \draw[->] (1) -- (5)  node[at end, below left] {$u$};

            \coordinate (6) at (0, 1); 
            \coordinate (7) at (1.25, 0.6);
            \coordinate (8) at (1.25, -0.5); 
            \coordinate (9) at (0.3, -1);
            \coordinate (10) at (-0.9, -0.6);
            \coordinate (11) at (-0.9, 0.6);
            \coordinate (12) at (1, 1.3);
            \coordinate (13) at (1, 0.3); 
            \coordinate (14) at (2.1, -0.15);

            \draw[color=purple] (6) -- (7) -- (8) -- (9) -- (10)-- (11) -- (6);
            \draw[color=purple, dashed] (6) -- (12) -- (13) -- (14) -- (8);

   \foreach \x in {(6), (7), (8), (9), (10), (11), (12), (13), (14)}{
        \fill \x circle[radius=2pt];
        }

            \coordinate (a) at (6, 0+0.5);
            \coordinate (b) at (6, 1+0.5);
            \coordinate (c) at (6, 2+0.5);
            \coordinate (d) at (6+0.87, -0.5+0.5);
            \coordinate (e) at (6+1.73, -1+0.5);
            \coordinate (g) at (6-0.87, -0.5+0.5);
            \coordinate (h) at (6-1.73, -1+0.5);
            \coordinate (i) at (6+1.73, 1+0.5);
            \coordinate (k) at (6-1.73, 1+0.5);
            \coordinate (j) at (6, -1.87+ 0.5);

            \draw (a) -- (b) -- (c);
            \draw (a) -- (e) -- (d);
            \draw (a) -- (g) -- (h);
            \draw (c) -- (i) -- (e) -- (j) -- (h) -- (k) -- (c);

            \node[above right] at (a) {$a$};
            \node[right] at (b) {$b$};
            \node[above] at (c) {$c$};
            \node[above right] at (d) {$d$};
            \node[below right] at (e) {$e$};
            \node[below] at (g) {$f$};
            \node[below left] at (h) {$g$};
            \node[above right] at (i) {$h$};
            \node[below] at (j) {$i$};
            \node[above left] at (k) {$j$};

               \foreach \x in {(a), (b), (c), (d), (e), (g), (h), (i), (j), (k)}{
        \fill \x circle[radius=2pt];
        }
            
        \end{tikzpicture}
        \caption{$\varphi_3': \Gamma_3' \rightarrow \RR^3$ contained in the tropical plane. Two edges in the positive span of $u$ and $e_3$ are hidden.}
        \label{fig:phi3}
    \end{figure}
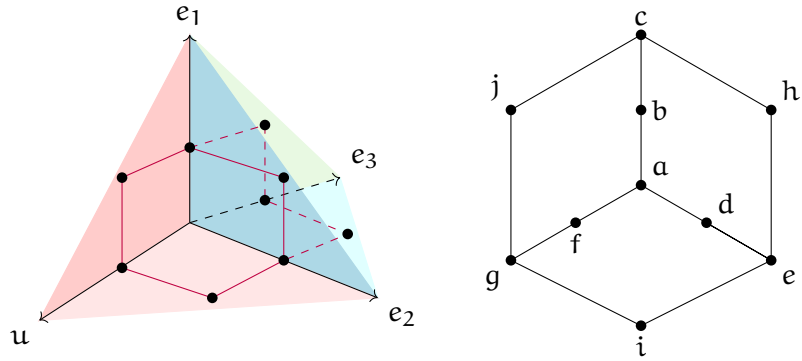

    \begin{table}[h]
        \centering
        \begin{tabular}{c|c}
        \toprule
        Vertex & Position \\
        \midrule
            a &  (-1, -1, -1)\\
            b &  (-1, -1, 1)\\
            c &  (0, 0, 1) \\
            d &  (-1, 1, -1)\\
            e &  (0, 1, 0)\\
            &\\
        \bottomrule
        \end{tabular}
        \begin{tabular}{c|c}
        \toprule
        Vertex & Position \\
        \midrule
            f &  (1, -1, -1)\\
            g &  (1, 0, 0)\\
            h &  (0, 1, 1)\\
            i &  (1, 1, 0)\\
            j &  (1, 0, 1)\\
            &\\
        \bottomrule
        \end{tabular}
        \caption{Positions of the labelled vertices.}
        \label{tab:phi3_pos}
    \end{table}

    By a direct check, the subcurve $(\Gamma_3', \varphi_3')$ has deformation space of dimension $4$ which is $1$ greater than the expected $3$. It has both irreducible and indecomposable superabundancy. 

    The balancing condition is not met at vertices $b, d, f, h, i$, and $j$. We may add legs at these vertices to obtain a tropical curve $(\Gamma_3, \varphi_3)$. 
\end{example}

\begin{remark}
    Intersecting the image of $\varphi_3$ with a tropical plane generically gives $4$ intersections (see Figure \ref{fig:phi3_int}). This mirrors the canonical embedding, which has degree $2g - 2 = 4$. This is not true of all tropical curves with core neighbourhood $\varphi_3'$ since we may add arbitrarily many trees. 

  \begin{figure}[h]
        \centering
        \begin{tikzpicture}
            \coordinate (1) at (0, 0); 
            \coordinate (2) at (0, 5);
            \coordinate (3) at (4, 1.2);
            \coordinate (4) at (5, -2);
            \coordinate (5) at (-4, -2.6);

            \filldraw[color=lightgreen] (1) -- (2) -- (3) -- cycle;
            \filldraw[color=lightcyan] (1) -- (3) -- (4) -- cycle;
            \filldraw[color=lightblue] (1) -- (4) -- (2) -- cycle;
            \filldraw[color=lightred] (1) -- (5)-- (2) -- cycle;
            \filldraw[color=lightpink] (1) -- (5) -- (4) -- cycle;
            
            \draw[->] (1) -- (2) node[at end, above] {$e_1$};
            \draw[->, dashed] (1) -- (3) node[at end, above right] {$e_3$};
            \draw[->] (1) -- (4) node[at end, below right] {$e_2$};
            \draw[->] (1) -- (5)  node[at end, below left] {$u$};

            \coordinate (6) at (0, 2); 
            \coordinate (7) at (2.5, 1.2);
            \coordinate (8) at (2.5, -1); 
            \coordinate (9) at (0.6, -2);
            \coordinate (10) at (-1.8, -1.2);
            \coordinate (11) at (-1.8, 1.2);
            \coordinate (12) at (2.1, 2.6);
            \coordinate (13) at (2.1, 0.6); 
            \coordinate (14) at (4.4, -0.3);

            \coordinate (a) at (1, 0.7);
            \coordinate (b) at (1, 3.6); 
            \coordinate (c) at (4, -0.6);
            \coordinate (d) at (0.5, -0.2); 
            \coordinate (e) at (-2.9, -2.5); 
            \coordinate (f) at (4.1, 0.9);

            \coordinate (i) at (-1.25, -1.4);
            \coordinate (j) at (2.5, 0.05);
            \coordinate (k) at (1, 1.65);
            \coordinate (l) at (2.6, 0.4);

            \draw[color=gray] (b) -- (a) -- (d) -- (e) ;
            \draw[color= gray] (a) -- (c);
            \draw[color=gray, dashed] (d) -- (f);

            \draw[color=purple, dashed] (6) -- (12) -- (13) -- (14) -- (8);

   \foreach \x in {(l)}{
        \fill \x circle[radius=2pt];
        }

            \draw[color=purple] (6) -- (7) -- (8) -- (9) -- (10)-- (11) -- (6);

   \foreach \x in {(i), (j), (k)}{
        \fill \x circle[radius=2pt];
        }

        \end{tikzpicture}
        \caption{Intersecting the image of $\varphi_3$ with a tropical plane.}
        \label{fig:phi3_int}
    \end{figure}
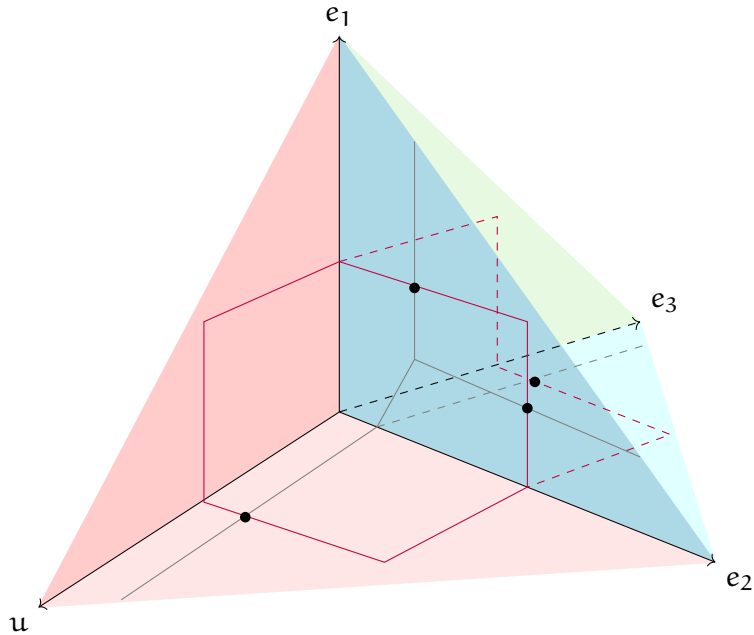
\end{remark}

We give a brief summary of the trivalent parametrised tropical curves of genus $3$ with irreducible superabundancy. We may assume vertex connectivity of at least two, which gives us two possible smoothed graphs, shown in Figure \ref{fig:g3graphs}. Possible choices for the matrix $\eta$ (which we may choose without loss of generality) for these two smoothed graphs are given in Table \ref{tab:trivalentEta}. 

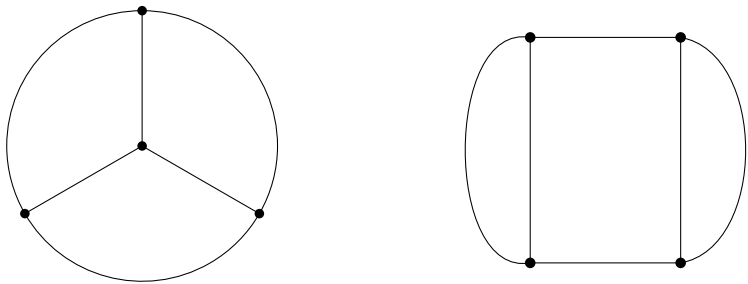
\begin{figure}[h]
    \centering
    \begin{tikzpicture}[scale=0.9]
    \draw (0, 0) arc (90:210:2);
    \draw(0, 0) arc (90:-30:2);
    \draw (-1.7320,-3) arc (210:330:2);
    \coordinate (1) at (0, 0);
    \coordinate (2) at (-1.7320, -3);
    \coordinate (3) at (1.7320, -3);
    \coordinate (4) at (0, -2);
    
    \draw (1) -- (4); 
    \draw (2) -- (4); 
    \draw (3) -- (4);
            
    \foreach \x in {(1), (2), (3), (4)}{
        \fill \x circle[radius=2pt];
        }
    
    \end{tikzpicture} 
    \hspace{2cm}
    \begin{tikzpicture}

    \coordinate (1) at (-1, 0);
    \coordinate (2) at (1, 0);
    \coordinate (3) at (-1, -3);
    \coordinate (4) at (1, -3);
    
    \draw (1) -- (2); 
    \draw (2) -- (4); 
    \draw (1) -- (3);
    \draw (3) -- (4); 
    \draw (2) to[out=-10,in=10] (4); 
    \draw (1) to[out=170,in=-170] (3); 
    
    \foreach \x in {(1), (2), (3), (4)}{
        \fill \x circle[radius=2pt];
        }
    \filldraw[white] (0,-3.5) ;
    \end{tikzpicture}
    \caption{Possible smoothed graphs for genus $3$ tropical curve.}
    \label{fig:g3graphs}
\end{figure}

\begin{table}[h]
	\centering
	\begin{tabular}{ccc} 
		\toprule
		Graph 1 & Graph 2 \\
		\midrule
		$\begin{bmatrix}
			1 & -1 & 0 & 1 & 0 & 0\\
			0 & 1 & -1 & 0 & 1 & 0\\
			-1 & 0 & 1 & 0 & 0 & 1\\
		\end{bmatrix}$ & $\begin{bmatrix}
			1 & 1 & 0 & 0 & 0 & 0\\
			0 & -1 & 1 & 0 & 1 & 1\\
			0 & 0 & 0 & 1 & -1 & 0\\
		\end{bmatrix}$ \\
		\bottomrule	
	\end{tabular}
        \caption{Choices of $\eta$ for the two possible genus $3$ smoothed graphs.}
        \label{tab:trivalentEta}
\end{table}

Suppose a genus 3 curve is superabundant. We have $\lambda^{(i)}, i=1,2,3$ with notation from above and since we are looking for irreducible superabundancy, we may assume that the $\lambda^{(i)}$ are non-zero. If all $\lambda^{(i)}$ are parallel to each other, the whole configuration lies inside an affine hyperplane, and we have planar superabundancy. Supposing otherwise, we consider when $\lambda^{(i)}$ are linearly independent, and linearly dependent. In each case we apply a linear transformation dual to sending the first three (resp. two) $\lambda^{(i)}$ to the first three (resp. two) unit normal vectors. 

The first graph with $\lambda^{(i)}$ linearly independent gives a union of planes that are extensions of the tropical plane in $\RR^3$. Example \ref{Exa:Genus3} is an example of this case. The linearly dependent case is always decomposable, since we may project to $\RR^2$ while preserving the arrangements of planes. In $\RR^2$, an illustrative example is given in Figure \ref{fig:triangle}.

\begin{figure}[h]
    \centering
        \begin{tikzpicture}
        \coordinate (1) at (0, 0);
        \coordinate (2) at (2,0);
        \coordinate (3) at (0, -2);
        \coordinate (4) at (0.5, -0.5);
        \draw (1) -- (2) -- (3) -- (1);
        \draw (1) -- (4); 
        \draw (2) -- (4); 
        \draw (3) -- (4);
        \foreach \x in {(1), (2), (3), (4)}{
        \fill \x circle[radius=2pt];
        }
    \end{tikzpicture}
    \caption{A superabundant subcurve of genus $3$ with underlying smoothed graph $1$.}
    \label{fig:triangle}
\end{figure}
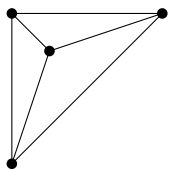

Possible superabundant curves for the second graph, with $\lambda^{(i)}$ linearly dependent and independent are shown in Figure \ref{fig:hyperbolic}. In these examples, we have a $2:1$ map to a tree. This is not true in general since we may add edges while preserving superabundancy.

\begin{figure}[h]
    \centering
    \tikzset{every picture/.style={line width=0.75pt}} 

\begin{tikzpicture}[x=0.75pt,y=0.75pt,yscale=-1,xscale=1]

\draw [color=blue  ,draw opacity=1 ]   (48.67,48.33) -- (122.47,159.47) ;
\draw [color=blue  ,draw opacity=1 ]   (209.13,48.8) -- (144.47,158.8) ;
\draw    (67.33,77) -- (191.8,78.13) ;
\draw    (91.23,112.57) -- (171.47,112.8) ;
\draw    (191.8,78.13) -- (213.13,90.8) ;
\draw    (171.47,112.8) -- (192.8,125.47) ;
\draw    (48.47,94.13) -- (67.33,77) ;
\draw    (72.37,129.7) -- (91.23,112.57) ;
\draw    (72.37,129.7) -- (48.47,94.13) ;
\draw    (72.37,113.7) -- (48.47,78.13) ;
\draw    (192.8,125.47) -- (213.13,90.8) ;
\draw    (192.13,112.13) -- (212.47,77.47) ;
\draw    (171.47,112.8) -- (192.13,112.13) ;
\draw    (191.8,78.13) -- (212.47,77.47) ;
\draw    (48.47,78.13) -- (67.33,77) ;
\draw    (72.37,113.7) -- (91.23,112.57) ;
\draw   (83.67,148.13) -- (238.47,148.13) -- (172.13,184.13) -- (17.33,184.13) -- cycle ;
\draw    (48.47,78.13) -- (21.13,63.47) ;
\draw    (72.37,113.7) -- (47.13,148.8) ;
\draw    (243.8,58.8) -- (212.47,77.47) ;
\draw    (247.13,87.47) -- (213.13,90.8) ;
\draw    (222.47,127.47) -- (192.13,112.13) ;
\draw    (199.13,142.13) -- (192.8,125.47) ;
\draw    (71.8,146.13) -- (72.37,129.7) ;
\draw    (25.8,94.13) -- (48.47,94.13) ;
\draw [color=red  ,draw opacity=1 ]   (122.47,159.47) -- (144.47,158.8) ;
\draw [color=red  ,draw opacity=1 ]   (144.47,158.8) -- (204.47,153.47) ;
\draw [color=red ,draw opacity=1 ]   (144.47,158.8) -- (181.13,168.8) ;
\draw [color=red  ,draw opacity=1 ]   (87.13,152.8) -- (122.47,159.47) ;
\draw [color=red  ,draw opacity=1 ]   (122.47,159.47) -- (73.13,166.13) ;
\draw [color=blue  ,draw opacity=1, dashed]   (133.8,176.8) -- (122.47,159.47) ;
\draw [color=blue  ,draw opacity=1, dashed]  (133.8,176.8) -- (144.47,158.8) ;
\draw [->]   (130.47,120.13) -- (130.47,140.8) ;

\draw    (368.67,78) -- (465.8,78.8) ;
\draw    (335.8,64.8) -- (368.67,78) ;
\draw    (368.67,78) -- (351.13,88.13) ;
\draw    (501.13,90.13) -- (465.8,78.8) ;
\draw    (485.8,70.13) -- (465.8,78.8) ;
\draw    (370,117.33) -- (467.13,118.13) ;
\draw    (337.13,104.13) -- (370,117.33) ;
\draw    (370,117.33) -- (352.47,127.47) ;
\draw    (502.47,129.47) -- (467.13,118.13) ;
\draw    (487.13,109.47) -- (467.13,118.13) ;
\draw    (335.8,64.8) -- (337.13,104.13) ;
\draw    (351.13,88.13) -- (352.47,127.47) ;
\draw    (485.8,70.13) -- (487.13,109.47) ;
\draw    (501.13,90.13) -- (502.47,129.47) ;
\draw    (335.8,64.8) -- (324.6,52.8) ;
\draw    (324.6,113.47) -- (337.13,104.13) ;
\draw    (337.93,140.13) -- (352.47,127.47) ;
\draw    (329.93,74.13) -- (351.13,88.13) ;
\draw    (485.8,70.13) -- (500.6,54.13) ;
\draw    (487.13,109.47) -- (505.93,114.8) ;
\draw    (501.13,90.13) -- (515.27,76.8) ;
\draw    (502.47,129.47) -- (511.27,142.8) ;
\draw   (336.6,150.4) -- (558.23,150.4) -- (503.2,183.73) -- (281.57,183.73) -- cycle ;
\draw [color=red  ,draw opacity=1 ]   (371.33,166.67) -- (468.47,167.47) ;
\draw [color=red ,draw opacity=1 ]   (338.47,153.47) -- (371.33,166.67) ;
\draw [color=red  ,draw opacity=1 ]   (371.33,166.67) -- (321.27,174.13) ;
\draw [color=red ,draw opacity=1 ]   (505.93,175.47) -- (468.47,167.47) ;
\draw [color=red ,draw opacity=1 ]   (507.93,156.8) -- (468.47,167.47) ;
\draw[->]    (416.47,123.47) -- (416.47,144.13) ;

\end{tikzpicture}
    \caption{Superabundant curves of genus $3$ with underlying smoothed graph 2.}
    \label{fig:hyperbolic}
\end{figure}
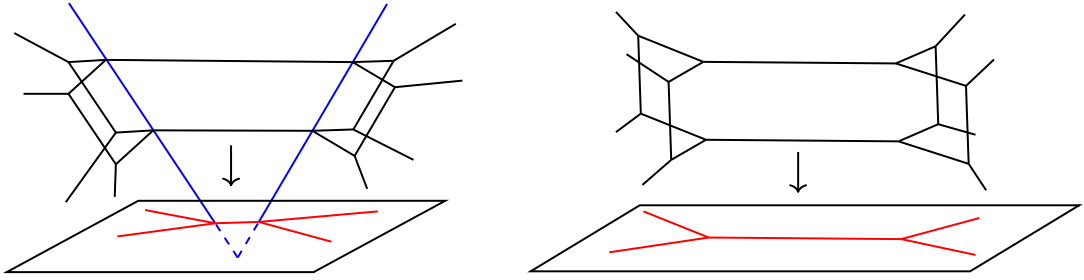

\begin{remark}
     We expect graphs of tree-gonality less than $ \lfloor (g+3)/2 \rfloor$ to also give examples of superabundant tropical curves.  
\end{remark}

\section{A genus 4 example}
\label{section3} \label{sec: genus-4}

While the method outlined in the section above generalises to higher genues, it becomes quickly difficult to track affine subspaces and their intersections. However, the method does suggest the existence of many ``new'' examples in every genus. We give an example in genus $4$. 

\begin{example}
The subcurve $\varphi_4': \Gamma_4 \rightarrow \RR^4$ is constructed from the smoothed graph which is a trivalent graph with 6 vertices, as shown in Figure \ref{fig:smoothed4}. Label the vertices as shown.

\begin{figure}[h]
        \centering
        \begin{tikzpicture}
            \coordinate (a) at (0, 2);
            \coordinate (b) at (+1.73, 1);
            \coordinate (c) at (1.73, -1);
            \coordinate (d) at (0, -1.87);
            \coordinate (e) at (-1.73, -1);
            \coordinate (f) at (-1.73, 1);

            \draw[red] (a) -- (b);
            \draw[red] (c) -- (f);
            \draw[red] (d) -- (e);
            \draw[blue] (a) -- (d);
            \draw[blue] (c) -- (b);
            \draw[blue] (f) -- (e);
            \draw[green] (a) -- (f);
            \draw[green] (e) -- (b);
            \draw[green] (d) -- (c);

            \node[above] at (a) {$a$};
            \node[above right] at (b) {$b$};
            \node[below right] at (c) {$c$};
            \node[below] at (d) {$d$};
            \node[below left] at (e) {$e$};
            \node[above left] at (f) {$f$};

               \foreach \x in {(a), (b), (c), (d), (e), (f)}{
        \fill \x circle[radius=2pt];
        }
            
        \end{tikzpicture}
        \caption{A $3$ regular graph with $6$ vertices.}
        \label{fig:smoothed4}
    \end{figure}
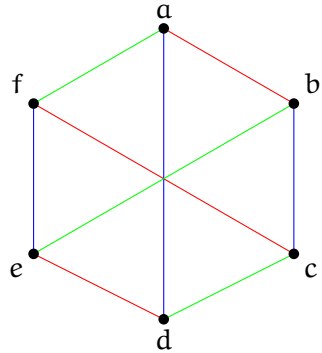

The curve is contained in a product $L^\trop\times Q^\trop\subset \mathbb R^4$, where $L^\trop$ is a tropical line in $\RR^2$ and $Q^\trop$ is a tropical conic in $\RR^2$. The projections to $L^\trop$ and $Q^\trop$ are given in Figure \ref{fig:phi4}. Each segment consists of four edges. The segment from $a$ to $b$ has edges with slopes: $(0,1,0,0)$ (go up a ray in $L^\trop$), $(0, 0, 0, 1)$, $(0, 0, 1, 0)$ (travel along edges of $Q^\trop$), $(0, -1, 0, 0)$ (return down ray in $L^\trop$). The others are analogous. 

    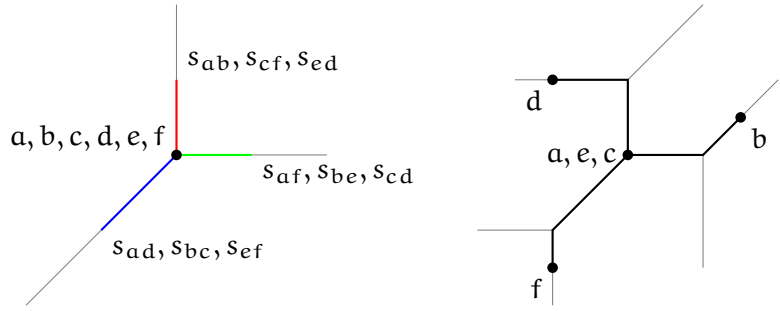
\begin{figure}[h]
        \centering
        \begin{tikzpicture}
            \coordinate (1) at (0, 0);
            \coordinate (2) at (1, 0);
            \coordinate (3) at (0, 1);
            \coordinate (4) at (2, 1);
            \coordinate (5) at (1, 2);
            \coordinate (6) at (1, -1.5);
            \coordinate (7) at (-1.5, 1);
            \coordinate (8) at (-1, -1);
            \coordinate (9) at (-2, -1);
            \coordinate (10) at (-1, -2);
            \coordinate (b) at (1.5, 0.5);
            \coordinate (d) at (-1, 1);
            \coordinate (f) at (-1, -1.5);

            \draw[color=gray] (1) -- (2) -- (4);
            \draw[color=gray] (1) -- (3) -- (5);
            \draw[color=gray] (2) -- (6);
            \draw[color=gray] (3) -- (7);
            \draw[color=gray] (1) -- (8) -- (9);
            \draw[color=gray] (8) -- (10);
            \draw[thick] (1) -- (2) -- (b);
            \draw[thick] (1) -- (3) -- (d);
            \draw[thick] (1) -- (8) -- (f);

            \node[left] at (1) {$a, e, c$};
            \node[below right] at (b) {$b$};
            \node[below left] at (d) {$d$};
            \node[below left] at (f) {$f$};

               \foreach \x in {(1), (b), (d), (f)}{
        \fill \x circle[radius=2pt];
        }

            \coordinate (s) at (-6, 0);
            \coordinate (t) at (-6, 2);
            \coordinate (u) at (-4, 0);
            \coordinate (v) at (-8, -2);
            \coordinate (w) at (-6, 1);
            \coordinate (x) at (-5, 0);
            \coordinate (y) at (-7, -1);

            \draw[color=gray] (s) -- (t);
            \draw[color=gray] (s) -- (u);
            \draw[color=gray] (s) -- (v);
            \draw[thick, color=red] (s) -- (w);
            \draw[thick, color=green] (s) -- (x);
            \draw[thick, color=blue] (s) -- (y);

            \node[above left] at (s) {$a, b, c, d, e, f$};
            \node[above right] at (w) {$s_{ab}, s_{cf}, s_{ed}$};
            \node[below right] at (x) {$s_{af},s_{be},s_{cd}$};
            \node[below right] at (y) {$s_{ad}, s_{bc}, s_{ef}$};
            \fill (-6,0) circle[radius=2pt];
        \end{tikzpicture}
        \caption{Projection of $\varphi_4'(\Gamma_4')$ to $L^\trop$ and $Q^\trop$. Here $s_{v_0v_1}$ denotes a segment between two vertices $v_0, v_1$ in the smoothed graph (which consists of four edges).}
        \label{fig:phi4}
    \end{figure}

    We may extend to a tropical curve $\varphi_4: \Gamma_4 \rightarrow \RR^4$ by adding legs at vertices which are not balanced. The projections to $L^{\trop}$ and $Q^{\trop}$ are each harmonic of degree $3$. 
    
    The dimension of the deformation space of $(\Gamma_4, \varphi_4)$ is $21$, whereas the expected dimension is $20$. It has irreducible superabundancy, since removing any segment creates a segment which is not contained in a hyperplane. However, a projection to $L^{\trop}$ gives a superabundant curve, so it does not have indecomposable superabundancy. 

    %
    %

\end{example}

\section{Non-realizable curves} \label{sec:4}

\subsection{Moduli spaces with larger than expected dimension}\label{sec:bigalg}

Superabundancy is not sufficient for non-realizability, even generically for curves of the given combinatorial type, since the moduli space of the corresponding algebraic curves may also have dimension larger than expected. 

For example, the triple tuning fork $\varphi_2: \Gamma_2 \rightarrow \RR^2$ is superabundant with $\dim \operatorname{Def}(\Gamma_2) = 3$, but we will show that every tropical curve with this combinatorial type is realizable. Algebraically, we are looking for genus $2$ curves with $3$ markings mapping to $\PP^2$, such that the markings map to the different coordinate lines, and the tangency along each coordinate line is $2$. Since $C$ is smooth, such a map $C \rightarrow \PP^2$ corresponds to a rational map from $C$ to the dense open torus $(\CC^*)^2$ in $\PP^2$, defined away from the marked points. Equivalently, this is given by the two principal divisors $\operatorname{div}(x)$ and $\operatorname{div}(y)$, where $x$ and $y$ are the two coordinates on the first standard affine patch of $\PP^2$. Any such divisor pair gives $(\mathbb{C}^*)^2$ worth of rational maps. Let $\cM^\circ$ be the locus of maps $(C, p, q, r) \rightarrow \PP^2$ with the tangency conditions above, and let $D_2(A)$ be the set of $(C, p, q, r) \in \cM_{2, 3}$ for which such a map exists. Then the natural map $\cM^\circ \rightarrow D_2(A)$ is a $(\CC^*)^2$-bundle. 

\begin{remark}
    For the deformation space of tropical curves, we took the quotient by the action given by translation, which corresponds algebraically to scaling the rational maps, i.e. taking the quotient corresponds to going from $\cM^\circ$ to $D_2(A)$. 
\end{remark}

\begin{proposition}
    The space $D_2(A)$ is a $3$-dimentional locus in $\cM_{2, 3}$. 
\end{proposition}

\begin{proof}
    We will show that $D_2(A)$ surjects onto $\cM_2$ (which is three dimensional) under the forgetful map $\cM_{2, 3} \rightarrow \cM_2$. Now $D_2(A)$ consists of those curves $C$ where both $\cO_C(2p - 2r)$ and $\cO_C(2q - 2r)$ are trivial. All genus $2$ curves admit a $2:1$ map to $\PP^1$, which by Riemann--Hurwitz has $6$ branching points. Given any branching point $p$, have that $2p$ is the pullback of a single point in $\PP^1$. Hence, for any branching points $p, q$, we have that $2p \sim 2q$. Thus for any $C \in \cM_2$, we may choose $p, q, r$ to be branching points of the cover, and then $(C, p, q, r) \in D_2(A)$. 
\end{proof}

\begin{proposition}
    For $\ch k \neq 2$, any tropical curve with the combinatorial type of $\varphi_2$ is realizable. 
\end{proposition}

\begin{proof}
    The map from the triple tuning fork to $\RR^2$ factors through the tropical line. We can add a vertex on each leg of the tropical line so that vertices are mapped to vertices (see Figure \ref{fig:harmonic}). Let the resulting map be $f$.

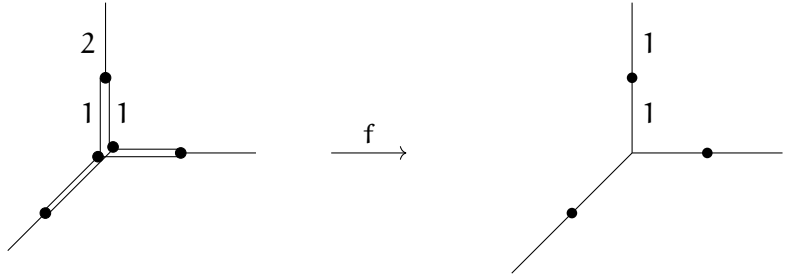
\begin{figure}[h]
    \centering
    \begin{tikzpicture}
        \filldraw [black] (-0.1, -0.05) circle (2pt);
        \filldraw [black] (0.1, 0.08) circle (2pt);
        \filldraw [black] (1, 0) circle (2pt); 
        \filldraw [black] (0, 1) circle (2pt); 
        \filldraw [black] (-0.8, -0.8) circle (2pt);
        \draw (-0.05,-0.05) -- (1, -0.05);
        \draw (-0.07,-0.05) -- (-0.07, 1);
        \draw (-0.1,-0.05) -- (-0.8,-0.75);
        \draw (0.1, 0.05) -- (1, 0.05);
        \draw (0.05, 0.05) -- (0.05, 1);
        \draw (0.1, 0.05) -- (-0.75, -0.8);
        \draw (-0.8, -0.8) -- (-1.3,-1.3);
        \draw (1, 0) -- (2, 0);
        \draw (0, 1) -- (0, 2);

        \draw[->] (3, 0) -- (4, 0);

        \coordinate (1) at (7, 0);
        \coordinate (2) at (9,0) ;
        \coordinate (3) at (7, 2);
        \coordinate (4) at (5.4, -1.6);
        \coordinate (5) at (8, 0);
        \coordinate (6) at (7, 1);
        \coordinate (7) at (6.2, -0.8);
        \node[above left] at (0,1.2) {$2$};
        \node[below left] at (0, 0.8) {$1$};
        \node[below right] at (0, 0.8) {$1$};
        \node[above right] at (7, 1.2) {$1$};
        \node[below right] at (7, 0.8) {$1$};
        \node [above] at (3.5,0) {$f$};
        
        \draw (1) -- (2);
        \draw (1) -- (3) ;
        \draw (1) -- (4);

        \foreach \x in {(5), (6), (7)}{
        \fill \x circle[radius=2pt];
        }

    \end{tikzpicture}
    \caption{Morphism from $\Gamma_2$ to tropical line satisfying local ``Riemann--Hurwitz'' condition.}
    \label{fig:harmonic}
\end{figure}

    By the results of Amini--Baker--Brugall\'e--Rabinoff \cite[Proposition 3.3]{ABBR15}, $f$ is realizable as a map between algebraic curves if it is locally realizable at vertices of $\Gamma_2$.\footnote{Here we require $\ch k \neq 2$ since the ramification degrees should not divide the $\ch k$. Otherwise, the result will extend to positive characteristic.} For the inner vertices, $f$ is a local isomorphism. For the outer vertices, we have two edges of weight $1$ mapping to an edge of weight $1$, and an edge of weight $2$ mapping to another edge of weight $1$. This corresponds (without loss of generality) to a morphism between marked curves $(\PP^1, p , q, r) \rightarrow (\PP^1, 0 , \infty)$, with $p, q \mapsto 0$, $r \mapsto \infty$, and ramification degrees $1, 1, 2$ respectively. Such a morphism exists, and the result follows. 
\end{proof}

\subsection{Non-realizability results}

We finish with a proof of Theorem \ref{thmA}, which shows that the superabundant examples we have given do in fact give classes of non-realizable curves. Tropicalizations of moduli spaces are constructed in \cite{abramovich_tropicalization_2012}. We will use the following result of Ulirsch:

\begin{theorem}
    \cite[Theorem 1.1]{U15} Given a closed subset $Y$ of a log regular variety $X$, we have a tropical variety $\operatorname{Trop}_X(Y)$ associated to $Y$ relative to $X$. $\operatorname{Trop}_X(Y)$ admits the structure of a finite rational polyhedral cone complex of dimension $\leq \dim Y$.
\end{theorem}

Consider $\mathcal{M}_g(A)$, where $A$ is a given $r \times n$ degree matrix, and ${\cM}_g(A)$ is the moduli space of 
\begin{enumerate}[(i)]
    \item a curve $C$ of genus $g$; 
    \item $n$ points on $C$; 
    \item $r$ morphisms $f_i:C \rightarrow \PP$, with order of vanishing of $f_i$ at $p_j$ given by $A_{ij}$.
\end{enumerate}
This is precisely the interior of the space of logarithmic stable maps to the toric variety $(\PP^1)^r$ with the toric logarithmic structure \cite{AC14, GS13}. Note that it is equivalently the space of such maps to any toric variety of dimension $r$ with the toric logarithmic structure, as the interior of the moduli space does not depend on the precise toric variety \cite{AW18}. In general, $\dim \cM_g(A)$ may be larger than expected, but for suitable choice of matrix $A$, we can calculate the dimension exactly.

\begin{proposition} \label{lem:dim}
    Let $A$ be the matrix whose columns are given by $d$ copies of $e_1$, $e_2$, \ldots, $e_r$ (for standard basis vectors $e_i \in \ZZ^r$), followed by $d$ copies of $-(e_1 + \ldots + e_r)$.
    If $d > 2g-2$, then the moduli space $\mathcal{M}_g(A)$ has dimension $(r+1)d + (r-3)(1-g)$.     
\end{proposition}

\begin{proof}
    We have a morphism $\phi: \cM_g(A) \rightarrow \cM_g(\PP^r, d)$ given as follows. Given an element 
    $$(C, p, f_1, \ldots, f_r) \in \cM_g(A)$$ 
    we obtain a rational function $C \dashrightarrow \mathbb{G}_m^r$ by sending $p \mapsto (f_1(p),\ldots,f_r(p))$. Since the curve is smooth, this can be completed to a morphism $C \rightarrow \PP^r$. By the choice of $A$, this has degree $d$, and we obtain a point of $\cM_g(\PP^r, d)$. So $\phi$ is forgetting the marked points $p_1, \ldots, p_n$. 
    
    Claim that $\phi$ is \'etale. We use the infinitesimal lifting property. The image of $\cM_g(A)$ in $ \cM_g(\PP^r, d)$ consists of $(C, f: C \rightarrow \PP^r)$ such that $f$ meets the coordinate axes of $\PP^r$ transversely. A lift of $(C, f) \in \cM_g(\PP^r, d)$ is given by labelling the points of $C$ which are mapped to coordinate axes of $\PP^r$. A deformation of $f$ induces a deformation of these distinct points which is unique once a labelling has been chosen.


    It suffices to calculate the dimension of $\cM_g(\PP^r, d)$. This is a standard result that we will review. The forgetful map $\mathcal{M}_g(\PP^r, d) \rightarrow \mathcal{M}_g$ to the space of curves of genus $g$ factorizes through $Pic^d(\mathcal{C}/\mathcal{M}_g)$, the universal Picard variety over the moduli space of curves. The fiber over any point $[X] \in \cM_g$ is $Pic^d(X)$, the line bundles of degree $d$ up to isomorphism. 

    
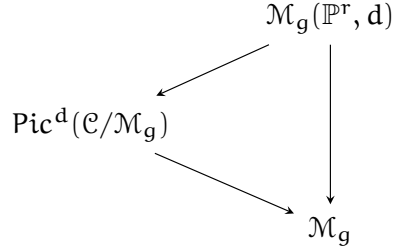
\begin{center}
        \begin{tikzcd}
        & {\mathcal{M}_g(\PP^r, d)} \arrow[dd] \arrow[ld] \\
        Pic^d(\mathcal{C}/\mathcal{M}_g) \arrow[rd] & \\
        & \mathcal{M}_g                                         
        \end{tikzcd}
    \end{center}

    The dimension of $\mathcal{M}_g$ is known to be $3g-3$, while the dimension of $Pic^d(\mathcal{C}/\mathcal{M}_g)$ is $4g-3$. The tangent space of a scheme $X$ is comprised of the $k[\epsilon]/(\epsilon^2) = k[\epsilon]$-valued points. This corresponds to the first order deformations, i.e. flat $k[\epsilon]$-scheme $\mathcal{X}$ with an isomorphism $X \cong \mathcal{X} \times_{k[\epsilon]} \spec k$. 
    
    
    A point in $\mathcal{M}_g(\PP^r, d)$ is given by $[C, L, f_0,\ldots,f_r \in H^0(C, L)]$, up to scaling all of the $f_i$, where we want $L$ to correspond to a divisor of $n$ distinct points, $f_i$ to be non-zero at these points, and to have distinct zeros elsewhere. Since these are open conditions, they do not change the dimension. We want to understand the \emph{vertical} tangent space at this point, i.e. which tangent vectors  go to zero when composed with the map to $Pic(\mathcal{C}/\mathcal{M}_g)$. These correspond to those deformations which fix the curve and the line bundle.

    Such deformations are given by $\tilde{f}_i \in H^0(C,L) \otimes k[\epsilon]$, so are of the form $f_i + \epsilon g_i$ with $g_i \in H^0(C,L)$. The $g_i$ that scale the functions give a trivial deformation since we are working in projective space. Hence the vertical tangent space to $p \in \mathcal{M}_g(\PP^r, d)$ has dimension $\oplus_{i=1}^r H^0(C, L)-1$. If $d > 2g-2$, then (by Riemann Roch) this is constant, equal to $(r+1)(d-g+1) - 1$. Putting this together, we obtain that the dimension of $\mathcal{M}_g(\PP^r, d)$ is 

    $$4g-3+(r+1)(d-g+1)-1$$
    which rearranges to give the desired result. 
\end{proof}

Consider the forgetful morphism $\cM_g(A) \rightarrow X = \bar{\cM}_{g,n}$. Its image, say $Y$, is closed in $X$, and the fibers over the image are $r$ dimensional (the group $\mathbb{G}_m^r$ acts on them by translation). So by Lemma \ref{lem:dim}, 
$$\dim Y = \dim \cM_g(A) - r = (r+1)d +(r-3)(1-g)-r$$
Suppose we have a parametrized tropical curve $\varphi$ of degree $d$ (so $d$ rays in each of the directions $e_1, \ldots , e_r$ and $-(e_1+\ldots +e_r)$) which contains $\varphi_3$ (resp. $\varphi_4$) as a core neighbourhood and has degree at least 5 (resp. 7). Then $\varphi$ is superabundant. For these values of $d$, we have that $\dim Y$ is the expected dimension, which is now strictly less than the dimension of the deformation space of $\varphi$. By \cite[Theorem 1.1]{U15} we have
$$\dim \trop_X (Y) \leq \dim Y< \dim \mathsf{Def}(\varphi)$$ 
and so the realizable curves form a cone with dimension strictly less than that of the ambient space. Hence a generic deformation of $\varphi$ is not realizable, as required.

\bibliographystyle{siam} 
\bibliography{SuperabundantConstructions}

\end{document}